\documentclass[12pt]{amsart}
\usepackage{amsmath,amssymb,amsfonts,latexsym,amscd,psfrag,mathabx,graphicx,mathrsfs,stmaryrd}
\usepackage[square,sort,comma,numbers]{natbib}
\usepackage{tikz}
\usepackage{verbatim}
\usepackage{tikz}
\usepackage[bottom]{footmisc}
\usetikzlibrary{shapes,shadows,calc}
\usepackage{footmisc}
\usepackage{epsfig}
\usepackage{float}
\usepackage{listings}
\usepackage{algorithm}
\usepackage{algorithmic}

\usepgflibrary{arrows}
\usetikzlibrary{arrows, decorations.markings, calc, fadings, decorations.pathreplacing, , decorations.shapes, patterns, decorations.pathmorphing, positioning}
\tikzset{nodde/.style={circle,draw=blue!50,fill=pink!80,inner sep=4.2pt}}
\tikzset{noddee/.style={circle,draw=black,fill=black,inner sep=2pt}}

\newcommand{\lebn}

\theoremstyle{plain}
\newtheorem{proposition}[equation]{Proposition}
\newtheorem{theorem}[equation]{Theorem}

\newtheorem{corollary}[equation]{Corollary}
\newtheorem{lemma}[equation]{Lemma}

\theoremstyle{definition}
\newtheorem{problem}[equation]{Problem}
\newtheorem{definition}[equation]{Definition}
\newtheorem{remark}[equation]{Remark}

\newtheorem{question}[equation]{Question} 
 
\numberwithin{equation}{section}

\newcommand{\N}{\mathbb{N}}

\newcommand{\PD}{\mathcal{P}\mathcal{D}}

\newcommand{\ID}{\mathcal{I}\mathcal{D}}

\newcommand{\E}{\mathcal{E}}
\newcommand{\B}{\mathcal{B}}

\newcommand{\CE}{\mathcal{C}}

\newcommand{\FE}{\mathcal{F}}
\newcommand{\RE}{\mathcal{R}}

\newcommand{\GE}{\mathcal{G}}
\newcommand{\LE}{\mathcal{L}}

\newcommand{\IE}{\operatorname{Ind}}
\newcommand{\HE}{\mathcal{H}}

\newcommand{\A}{\mathcal{A}}
\newcommand{\NE}{\mathcal{N}}

\newcommand{\kk}{\Bbbk}

\newcommand{\hf}{\mathfrak{h}}

\newcommand{\Ind}{\operatorname{Ind}}

\newcommand{\lk}{\operatorname{lk}}

\newcommand{\pd}{\operatorname{proj-dim}}
\newcommand{\cd}{\operatorname{cochord}}
\newcommand{\reg}{\operatorname{reg}}

\newcommand{\dis}{\operatorname{dist}}
\newcommand{\gi}{\operatorname{girth}}

\newcommand{\im}{\operatorname{im}}

\newcommand{\ie}{\operatorname{i}}
\newcommand{\prc}{\operatorname{prc}}
\newcommand{\imc}{\operatorname{imc}}
\newcommand{\us}{\operatorname{us}}
\newcommand{\ve}{\operatorname{ve}}
\newcommand{\ev}{\operatorname{ev}}
\newcommand{\ver}{\operatorname{vert}}

\newcommand{\spe}{\operatorname{Sp}}
\newcommand{\dom}{\operatorname{Dom}}

\newcommand{\D}{\Delta}

\begin{document}

\bibliographystyle{plain}

\title[Projective dimension of graphs and the regularity of bipartite graphs]{Projective dimension of (hyper)graphs and the Castelnuovo-Mumford regularity of bipartite graphs}
\author{T\" urker B\i y\i ko\u{g}lu and Yusuf Civan}

\address{}

\address{Department of Mathematics, Suleyman Demirel University,
Isparta, 32260, Turkey.}

\email{tbiyikoglu@gmail.com and yusufcivan@sdu.edu.tr }


\date{\today}

\thanks{Both authors are supported by T\" UB\. ITAK (grant no: 111T704), and the first author is
also supported by ESF EUROCORES T\" UB\. ITAK (grant no: 210T173)}


\begin{abstract}
We prove that $\pd(\HE)\leq \reg(\LE(\HE))$ holds for any (hyper)graph $\HE$, where $\LE(\HE)$ denotes the Levi graph (the incidence bipartite graph) of $\HE$. This in particular brings the use of regularity's upper bounds on the calculation of $\pd(\HE)$. When $G$ is just a (simple) graph, we prove that there exists an induced subgraph $H$ of $G$ such that
$\pd(G)=\reg(S(H))$, where $\LE(H)=S(H)$ is the subdivision graph of $H$. Moreover, we show that known upper bounds on $\pd(G)$ involving domination parameters are in fact upper bounds to $\reg(S(G))$. By way of application, we prove that $\pd(\dom(G))=\Gamma(G)$ for any graph $G$, where
$\dom(G)$ is the dominance complex of $G$ and $\Gamma(G)$ is the upper domination number of $G$. As a counterpart of prime graphs introduced for the regularity calculations of graphs, we call a connected graph $G$ as a \emph{projectively prime graph} over a field $\kk$ if $\pd_{\kk}(G-x)<\pd_{\kk}(G)$ holds for any vertex $x\in V(G)$. Such a notion allows us to create examples of graphs showing that most of the known lower and upper bounds on the projective dimension of graphs are far from being tight.
\end{abstract} 

\maketitle

\section{Introduction}\label{sect:intro}
Let $R=\kk[V]$ be a polynomial ring with a finite set $V$ of indeterminates over a field $\kk$. When $I$ is a monomial
ideal in $R$, there are two central invariants associated to $I$, the regularity $\reg(I):=\max\{j-i\colon \beta_{i,j}(I)\neq 0\}$ and the projective dimension $\pd(I):=\max\{i\colon \beta_{i,j}(I)\neq 0\;\text{for some}\;j\}$, that in a sense, they measure the complexity of computing the graded Betti numbers $\beta_{i,j}(I)$ of $I$. In particular, if $I$ is a squarefree monomial ideal, these invariants are related by a well-known duality result of Terai~\cite{NT} stating that $\pd(I)=\reg(R/I^{\vee})$, where $I^{\vee}$ is the Alexander dual of $I$. 

The main purpose of the current work is to explore the nature of the Terai's duality from the combinatorial point of view. Recall that if $I$ is  minimally generated by square-free monomials $m_1,\ldots,m_r$, then the family $\D:=\{K\subseteq V\colon F_i\nsubseteq K\;\text{for any}\;i\in [r]\}$, where $F_i:=\{x_j\colon x_j|m_i\}$ is a simplicial complex on  $V=\{x_1,\ldots, x_n\}$ such that whose minimal non-faces exactly correspond to minimal generators of $I$. Under such an association, the ideal $I$ is said to be the \emph{Stanley-Reisner ideal} of the simplicial complex $\D$, denoted by $I=I_{\D}$, and the simplicial complex $\D_I:=\D$ is called the \emph{Stanley-Reisner complex} of the ideal $I$. Furthermore, the pair $\HE(I):=(V,\{F_1,\ldots,F_r\})$ is a (simple) hypergraph (known also as a \emph{clutter}) on the set $V$, and under such a correspondence, the ideal $I$ is called the \emph{edge ideal} of the hypergraph $\HE(I)$.

These interrelations can be reversible in the following way.
Let $\HE=(V,\E)$ be a hypergraph on $V$. Then the family of subsets of $V$ containing no edges of $\HE$ forms a simplicial complex $\IE(\HE)$ on $V$, the \emph{independence complex} of $\HE$, and the corresponding Stanley-Reisner ideal $I_{\HE}:=I_{\IE(\HE)}$ is exactly the edge ideal of $\HE$. Under such settings, the regularity of a simplicial complex is defined by $\reg(\D):=\reg(R/I_{\D})$ from which we set $\reg(\HE):=\reg(\Ind(\HE))$ for any (hyper)graph $\HE$.

In the topological combinatorial side, the classical combinatorial Alexander duality has already some fruitful consequences. For instance, Csorba~\cite{PC} proved that the independence complex $\Ind(S(G))$ of the subdivision graph $S(G)$ of any graph $G$ is homotopy equivalent to the suspension of the Alexander dual $\Ind(G)^{\vee}$. However, his result becomes a special case when we consider the result of Nagel and Reiner~\cite{NR} on the homotopy type of the suspension of an arbitrary simplicial complex. In detail, they show that if $\D$ is a simplical complex on $V$ with the set of facets $\FE_{\D}$, and $B(\D)$ is the bipartite graph that can be obtained by taking the bipartite complement of the Levi graph of $\D$ (or the incidence graph of the set system $(V,\FE_{\D})$), then the homotopy equivalence $\Ind(B(\D))\simeq \Sigma(\D)$ holds (see Section~\ref{sect:bir} for details). Even if the regularity of a simplicial complex is not a homotopy invariant, we show that it behaves well under such a homotopy equivalence, that is, we prove that
the inequality $\reg(\D)+1\leq \reg(B(\D))$ holds for any simplicial complex $\D$ (Theorem~\ref{cor:reg-bip-simp}).
When we combine such a result with the Terai's duality, we conclude that $\pd(\HE)\leq \reg(\LE(\HE))$ for any hypergraph $\HE$, where $\LE(\HE)$ is the Levi graph of $\HE$. In the case where
$\HE=H$ is just a (simple) graph, such an inequality becomes
$\pd(H)\leq \reg(S(H))$ that corresponds to the Csorba's special case in the algebraic language. In particular, we prove that every graph $H$ has an induced subgraph $H'$ such that $\pd(H)=\reg(S(H'))$. 

Such interrelations naturally bring the question of whether the known lower and upper bounds on the projective dimension of graphs and the regularity of (bipartite) graphs are comparable? Firstly, we show that the induced matching number of the bipartite graph $B(\D)$ is closely related to the Helly number of the simplicial complex $\D$ (Corollary~\ref{cor:im-bdelta}). On the other hand, we verify that upper bounds on $\pd(H)$ involving domination parameters of $H$ invented by Dao and Schweig in a series of papers~\cite{DS1,DS2,DS3} are in fact upper bounds to $\reg(S(H))$. Interestingly, the induced matching and cochordal cover numbers of $S(H)$ can be expressed in terms of the domination parameters of $H$ that may be of independent interest. Indeed, we show that the equalities $\im(S(H))=|H|-\gamma(H)$ and $\cd(S(H))=|H|-\tau(H)$ hold for any graph $H$ without any isolated vertex (Proposition~\ref{prop:imsg} and Corollary~\ref{cor:cochord-tau}), where $\gamma(H)$ is the domination number of $H$, whereas $\tau(H)$ is the independence domination number of $H$. 

As a counterpart of prime graphs introduced for the regularity calculations of graphs~\cite{BC1}, we call a connected graph
$G$ as a \emph{projectively prime graph} over a field $\kk$ if $\pd_{\kk}(G-x)<\pd_{\kk}(G)$ holds for any vertex $x\in V(G)$. Moreover, we say that
$G$ is a \emph{perfect} projectively prime graph if it is projectively prime graph over any field. In particular, we verify that the join of any two graphs is a perfect projectively prime graph. In one hand, such a notion turns the projective dimension calculations of graphs into generalized (weighted) induced matching problems, on the other hand, it allows us to create examples of graphs showing that most of the known lower and upper bounds on the projective dimension is far from being tight.

\section{Preliminaries}\label{sect:prel}

\subsection{(Hyper)graphs}\label{subsect:hyper}
By a (simple) graph $G$, we will mean a finite undirected graph without loops or multiple edges. If $G$ is a graph, $V(G)$ and $E(G)$ (or simply $V$ and $E$) denote its vertex and edge sets. If $U\subset V$, the graph induced on $U$ is written $G[U]$, and in particular, we abbreviate $G[V\backslash U]$ to $G-U$, and write $G-x$ whenever $U=\{x\}$. For a given subset $U\subseteq V$, the (open) neighbourhood of $U$ is
defined by $N_G(U):=\cup_{u\in U}N_G(u)$, where $N_G(u):=\{v\in V\colon uv\in E\}$, and similarly, $N_G[U]:=N_G(U)\cup U$ is the closed neighbourhood of $U$. Furthermore, if $F=\{e_1,\ldots,e_k\}$ is a subset of edges of $G$, we write $N_G[F]$ for the set $N_G[V(F)]$, where $V(F)$ is the set of vertices incident to edges in $F$. 
We write $G^{\circ}$ for the graph obtained from $G$ by removing isolated vertices (if any). 

Throughout $K_n$, $C_n$ and $P_n$ will denote the complete, cycle and path graphs on $n$ vertices respectively. 

We say that $G$ is $H$-free if no induced subgraph of $G$ is isomorphic to $H$. A graph $G$ is called \emph{chordal} if it is $C_r$-free for any $r>3$.  Moreover, a graph is said to be \emph{cochordal} if its complement is a chordal graph.

A \emph{hypergraph} is a (finite) set system $\HE=(V, \E)$, where $V$ is a set, the set of vertices of $\HE$, and $\E\subseteq 2^{V}$ is the set of edges of $\HE$. The hypergraphs that we consider do not have loops, that is, $\E$ can not contain singletons as edges. A hypergraph $\HE=(V, \E)$ is said to be \emph{simple} if $A\nsubseteq B$ and $B\nsubseteq A$ for any two edges $A,B\in \E$. Through our work, we only deal with simple hypergraphs. 

When $S\subseteq V$, the \emph{induced subhypergraph} of $\HE$ by the set $S$ is defined to be the hypergraph
$\HE[S]:=\{A\in \E\colon A\subseteq S\}$. 

Recall that a subset $M\subseteq \E$ is called a {\it matching} of $\HE$ if no two edges in $M$ share a common vertex. Moreover, a matching $M=\{F_1,\ldots,F_k\}$ of $\HE$ is an {\it induced matching} if the edges in $M$ are exactly the edges of the induced subhypergraph of $\HE$ over the vertices contained in $\cup_{i=1}^k F_i$, and the cardinality of a maximum induced matching is called the {\it induced matching number} of $\HE$ and denoted by $\im(\HE)$. 

\begin{definition}
For a given hypergraph $\HE=(V,\E)$, the bipartite graph $\LE(\HE)$ on $V\cup \E$ defined by $(x,F)\in E(\LE(\HE))$ if and only if $x\in F$ is called the \emph{Levi graph} (or the \emph{incidence graph}\footnote{It seems that this terminology is not standard, both stemming from incidence geometry.}) of $\HE$.
\end{definition}

If we drop the condition of being simple for hypergraphs, any bipartite graph appears as the Levi graph of some hypergraph, even if such a representation is not unique. Since we only deal with simple hypergraphs, the resulting bipartite graphs are of special types. 
Let $B$ be a bipartite graph with a bipartition $V(B)=X_0\cup X_1$. Then $B$ is called a \emph{Sperner bipartite graph} (or shortly an $\spe$-bipartite graph) on $X_i$, if $X_i$ contains no pair of vertices $u$ and $v$ such that $N_B(u)\subseteq N_B(v)$. 
Observe that a bipartite graph $B$ satisfies $B\cong \LE(\HE)$ for some (simple) hypergraph $\HE$ if and only if $B$ is a $\spe$-bipartite graph.

We also note that when $\HE=H$ is just a graph, its Levi graph $\LE(H)$ is known as the \emph{subdivision graph} of $H$, denoted by $S(H)$, which can be obtained from $H$ by subdividing each edge of $H$ exactly once. The characterization of subdivision graphs is rather simple, namely that there exists a graph $G$ such that
$B\cong S(G)$ if and only if $B$ is $C_4$-free and $\deg_B(v)=2$ for all $v\in X_i$ for some $i=0,1$. 

\subsection{Domination parameters of graphs}\label{sect:dompar}
We next recall the definitions of some domination parameters of graphs that will be in use throughout the sequel. We note that the relationship between projective dimension of graphs and various domination parameters firstly explored by Dao and Schweig in a series of papers~\cite{DS1,DS2,DS3}. However,  some of their newly defined notions turn out to be already known in the literature~\cite{ABZ,JWP,JL}.

A subset $A\subseteq V$ is called a \emph{dominating set} for $G$ if $N_G[A]=V(G)$, and the minimum size $\gamma(G)$ of a dominating set for $G$ is called the \emph{domination number} of $G$, while the maximum size $\Gamma(G)$ of a minimal dominating set of $G$ is known as the \emph{upper domination number} of $G$.
Moreover, the least cardinality $\ie(G)$ of an independent dominating set for $G$ is called the
\emph{independent domination number} of $G$.

When $S\subseteq V$,  a vertex $x\in V\setminus S$ is said to be a \emph{private neighbour} of a vertex $s\in S$ with respect to $S$ in $G$ if $N_G(x)\cap S=\{s\}$, and the set of private neighbours of $s\in S$ is denoted by $P_G^S(s)$.

A subset $F\subseteq E$ is called an \emph{edge dominating set} of $G$ if each edge of $G$ either belongs to $F$ or is incident to some
edge in $F$, and the \emph{edge domination number} $\gamma'(G)$ of $G$ is the minimum cardinality of an edge dominating set of $G$. We remark that $\gamma'(G)$ equals to the minimum size of a maximal matching for any graph $G$~\cite{YG}. 

For given subsets $X,Y\subseteq V$, we say that $X$ \emph{dominates} $Y$ in $G$ if $Y\subseteq N_G(X)$, and let $\gamma(Y,G)$ denote the least cardinality of a subset $X$ that dominates $Y$ in $G$.
$$\tau(G):=\max \{\gamma(A,G^{\circ})\colon A\subseteq V(G^{\circ})\;\textrm{is\;an\;independent\;set}\}$$
is called the \emph{independence domination number} of $G$.

\begin{remark}\label{rem:dao-sweig}
We note that the independence domination number $\tau(G)$ seems to first appear in the work of Aharoni, Berger and Ziv~\cite{ABZ} (see also~\cite{AS}) in which it is denoted by $\gamma^i(G)$.
\end{remark}
\subsubsection{Vertex-edge and edge-vertex dominations in graphs}
Most of the research on graph parameters involved domination mainly concentrates either sets of vertices dominating all other vertices or sets of edges dominating all other edges in graphs. On the contrary, Peters~\cite{JWP} has introduced vertex-edge and edge-vertex dominations in graphs. Since then there has been little research on such a mixing theory of graph dominations~\cite{BCHH,JL}. However, the recent results of Dao and Schweig~\cite{DS1,DS2,DS3} have validated the usefulness of such notions in the calculation of projective dimension of (hyper)graphs that we describe next.

A vertex $u\in V$ is said to \emph{vertex-wise dominate} an edge $e=xy\in E$, if $u\in N_G[e]$. A subset $S\subseteq V$ is called a \emph{vertex-wise dominating set}, if for any edge $e\in E$, there exists a vertex $s\in S$ that vertex-wise dominates $e$. Furthermore, a vertex-wise dominating set $S$ of $G$ is called \emph{minimal}, if no proper subset of $S$ is vertex-wise dominating for $G$. When $S$ is a minimal vertex-wise dominating set for $G$, every vertex in $S$ has a private neighbour in $E$, that is, $e$ is a \emph{vertex-wise private neighbour} of $s\in S$ if $s$ vertex-wise dominates $e$ while no vertex in $S\setminus \{s\}$ vertex-wise dominates the edge $e$ in $G$.

$(1)\;$ $\Upsilon(G):=\max\{|S|\colon S\;\textrm{is\;a\;minimal\;vertex-wise\;dominating\;set\;of\;}G\}$ is called the \emph{upper} \emph{vertex-wise domination number} of $G$.

$(2)\;$ $\beta(G):=\max\{|S|\colon S\;\textrm{is\;a\;minimal\;independent\;vertex-wise\;dominating\;set\;of\;}G\}$ is called the \emph{upper independent vertex-wise domination number} of $G$.

Observe that the inequality $\beta(G)\leq \Upsilon(G)$ holds for any graph $G$ as a result of the definitions.

An edge $e=xy$ of $G$ is said to \emph{edge-wise dominate} a vertex $u\in V$ if $u\in N_G[e]$. A subset $F\subseteq E$ is called an \emph{edge-wise dominating set}, if for any vertex $u\in V(G^{\circ})$, there exists an edge $f\in F$ that edge-wise dominates $u$. 
Moreover, a vertex $v$ is said to be an \emph{edge-wise private neighbour} of $f\in F$ if $f$ edge-wise dominates $v$ while no edge in $F\setminus \{f\}$ edge-wise dominates $v$.

$(3)$ $\epsilon(G):=\min\{|F|\colon F\subseteq E\;\textrm{is\;an\;edge-wise\;dominating\;set\;of\;}G\}$ is called the \emph{edge-wise domination number} of $G$.

We remark that the parameter $\epsilon(G)$ can also be considered as the minimum size of an edge-wise dominating matching of $G$ (compare to Theorem $73$ in~\cite{JL}).

\begin{remark}\label{rem:peters}
We note that Peters~\cite{JWP} chooses to call vertex-wise and edge-wise dominating sets by $\ve$-dominating and $\ev$-dominating sets respectively, and denotes parameters $\Upsilon(G)$ and $\epsilon(G)$ by $\Gamma_{\ve}(G)$ and $\gamma_{\ev}$ respectively. However, we prefer to follow the most recent notations of Dao and Schweig~\cite{DS1} in order to overcome the difficulty of readability caused by the prefixes $\ve$- and $\ev$-.
\end{remark}
\subsection{Simplicial complexes}\label{subsect:simpl}
An \emph{(abstract) simplicial complex} $\D$ on a finite set $V$ is a family of subsets of $V$ such that $\{v\}\in \D$ for all $v\in V$, and
if $F\in \D$ and $H\subset F$, then $H\in \D$.
The elements of $\D$ are called \emph{faces} of it; the \emph{dimension} of a face $F$ is $\textrm{dim}(F):=|F|-1$, and the \emph{dimension} of $\D$ is defined to be  $\textrm{dim}(\D):=\textrm{max}\{\textrm{dim}(F)\colon F\in \D\}$. The $0$ and $1$-dimensional faces of $\D$ are called \emph{vertices} and \emph{edges} while maximal faces are called \emph{facets}. In particular, we denote by $\FE_{\D}$, the set of facets of $\D$. For a given face $A\in \D$, the subcomplex $\lk_{\D}(A):=\{F\in \D\colon F\cap A=\emptyset\;\textrm{and}\;F\cup A\in \D\}$ is called the \emph{link} of $A$ in $\D$.

When $\D=\Ind(G)$ for some graph $G$, the existence of vertices satisfying some extra properties is useful when dealing with
the homotopy type:
\begin{theorem}\cite{AE, MT}\label{thm:hom-induction}
If $N_G(u)\subseteq N_G(v)$, then there is a homotopy equivalence $\IE(G)\simeq \IE(G-v)$. On the other hand,
if $N_G[u]\subseteq N_G[v]$, then the homotopy equivalence $\IE(G)\simeq \IE(G-v)\vee \Sigma \IE(G-N_G[v])$
holds.
\end{theorem}

In~\cite{BC1} we introduced prime simplicial complexes that constitute building blocks for the regularity calculations. In detail, a simplicial complex $\D$ on $V$ is said to be a \emph{prime simplicial complex} over a field $\kk$ provided that $\reg_{\kk}(\D-x)<\reg_{\kk}(\D)$ for any vertex
$x\in V$. In particular, a graph $G$ is called a \emph{prime graph} (over $\kk$) whenever $\Ind(G)$ is a prime simplicial complex.

The notion of the Alexander dual of a simplicial complex plays a prominent role on the calculation of the projective dimension of $\D$ due to the  Terai's duality. However, for its definition, we need to relax the definition of a simplicial complex first. A \emph{generalized simplicial complex} $\D$ on a vertex set $V$ is simply a family of subsets of $V$, closed under inclusion. In particular, $\ver(\D):=\{x\in V\colon \{x\}\in \D\}$ is called the set of \emph{actual vertices} of $\D$. Observe that when $\D=\Ind(G)$ for some graph $G$, then $v\notin \ver(\Ind(G)^{\vee})$ if and only if $v$ is incident to any edge $e\in E$, i.e., $G\cong K_{1,n}$ for some $n\geq 1$, where $v$ is the vertex with $\deg_G(v)=n$.

\begin{definition}
Let $\D$ be a (generalized) simplicial complex on $V$. Then the \emph{Alexander dual} of $\D$ is defined to the simplicial complex $\D^{\vee}:=\{F\subseteq V\colon V\setminus F\notin \D\}$.
\end{definition} 

Observe that the equality $(\D^{\vee})^{\vee}=\D$ always holds. Furthermore, we recall that when $\D=\Ind(\GE)$ for some (hyper)graph $\GE$, the minimal non-faces of $\Ind(\GE)^{\vee}$ exactly correspond to minimal vertex covers of $\GE$, where a subset $U\subseteq V$ is a \emph{vertex cover} of $\GE$
if $V\setminus U\in \Ind(\GE)$. For this reason, the edge ideal
$I^{\vee}_{\GE}:=I_{\Ind(\GE)^{\vee}}$ of $\GE^{\vee}:=\HE(\Ind(\GE)^{\vee})$ is also known as the \emph{vertex cover ideal} of $\GE$.


\section{Bounding the regularity of simplicial complexes from their Levi graphs}\label{sect:bir}

In this section, we prove that the regularity of any simplicial complex (or any hypergraph) can be bounded from above by the regularity of a bipartite graph. This, in particular, implies together with the Terai's duality that the regularity of the Levi graph $\LE(\HE)$ of any hypergraph $\HE$ provides an upper bound to the projective dimension of $\HE$. For that purpose, we first provide a local version of Nagel and Reiner's result~\cite{NR}. Along the way, we show that the induced matching number of the associated bipartite graph can be derived from the combinatorics of the underlying simplicial complex.

\begin{definition}
Let $\D$ be a simplicial complex on $V$. We define a bipartite graph $B_S(\D)$ for each $S\subseteq V$
by $V(B_S(\D)):=S\cup \FE_{\D}$ and $(s,F)\in E(B_S(\D))$ if and only if $s\notin F$ for any $s\in S$ and $F\in \FE_{\D}$, where $\FE_{\D}$ is the set of facets of $\D$.
In particular, we abbreviate $B_V(\D)$ to $B(\D)$.
\end{definition}

We note that if we consider $\D$ as a hypergraph $\HE(\D)=(V,\FE_{\D})$, the graph $B(\D)$ is just the bipartite complement of the Levi graph of $\HE(\D)$. 

The following theorem of Nagel and Reiner~\cite{NR} (see also Jonsson~\cite{JJ}) is our main motivation for the definition of $B(\D)$:

\begin{theorem}\label{thm:NR}~\cite{NR}
The homotopy equivalence $\Ind(B(\D))\simeq \Sigma(\D)$ holds for any simplicial complex $\D$, where $\Sigma(\D)$ is the (unreduced) suspension $\D$.
\end{theorem}

We first provide a local version of Theorem~\ref{thm:NR}.
\begin{lemma}\label{lem:max-local}
Let $\D$ be a simplicial complex. Then $\IE(B_S(\D))\simeq \IE(B(\D[S]))$ for any $S\subseteq V$.
\end{lemma}
\begin{proof}
Assume that $\FE_{\D[S]}=\{H_1,\ldots,H_k\}$, and define $\FE_i:=\{F\in \FE_{\D}\colon H_i\subseteq F\}$ for each $i\in [k]$,
$\A_S:=\{C\in \FE_{\D}\colon \emptyset\neq C\cap S\notin \FE_{\D[S]}\}$ and $\B_S:=\{D\in \FE_{\D}\colon D\cap S=\emptyset\}$.
Note that if $(s,H_i)\in E(B(\D[S]))$ for some $s\in S$ and $i\in [k]$, then $(s,F)\in E(B_S(\D))$ for any $F\in \FE_i$. 
Suppose otherwise $H_i\cup \{s\}\in \D$ that in turn implies that
$H_i\cup \{s\}\in \D[S]$. However, the latter is impossible due to the fact that $H_i$ is a facet of $\D[S]$. It then follows that 
\begin{equation*}
N_{B_S(\D)}(F)=N_{B_S(\D)}(F')=N_{B_S(\D[S])}(H_i)
\end{equation*}
for any $F,F'\in \FE_i$ and $i\in [k]$. If we fix a facet, say $F_i\in \FE_i$ for each $i\in [k]$,
the homotopy equivalence $\IE(B_S(\D))\simeq \IE(B_S(\D))-(\FE_i\setminus F_i))$ holds by Theorem~\ref{thm:hom-induction}.
In other words, we can remove the facets in $\FE_i$ one by one until there remains only a single facet without altering the homotopy type.

On the other hand, if $C\in \A_S$, then $C\cap S$ is a non empty face of $\D[S]$ which is not maximal. So, it is contained by a facet, say $H_i$. However, the containment $C\cap S\subseteq H_i$ forces that $N_{B_S(\D)}(F_i)\subseteq N_{B_S(\D)}(C)$, where $F_i\in \FE_i$ is the facet that we fixed earlier. Therefore, the removal of any such vertex from $\IE(B_S(\D))$ does not alter the homotopy type. Similarly, if $D\in \B_S$, then we have $N_{B_S(\D)}(F_i)\subseteq N_{B_S(\D)}(D)$ for any $i\in [k]$ so that $\IE(B_S(\D))\simeq \IE(B_S(\D)-\B_S)$.

Finally, if we define $\CE_S:=\cup_{i=1}^k (\FE_i\setminus F_i)\cup \A_S\cup \B_S$, we conclude that
\begin{equation*}
\IE(B_S(\D))\simeq \IE(B_S(\D)-\CE_S)\cong \IE(B_S(\D)[S\cup\{F_1,\ldots,F_k\}]).
\end{equation*}
However, the graph $B_S(\D)[S\cup\{F_1,\ldots,F_k\}]$ is isomorphic to $B_S(\D[S])$ that proves the claim.
\end{proof}

\begin{corollary}\label{cor:max-local}
$\IE(B_S(\D))\simeq \Sigma(\D[S])$ for any $S\subseteq V$.
\end{corollary}
\begin{proof}
Since $\Sigma(\D[S])\simeq \IE(B_S(\D[S]))$ by Theorem~\ref{thm:NR}. Thus, the result follows from Lemma~\ref{lem:max-local}.
\end{proof}

\begin{theorem}\label{cor:reg-bip-simp}
The inequality $\reg(B(\D))\geq \reg(\D)+1$ holds for any simplicial complex $\D$.
\end{theorem}
\begin{proof}
Suppose that $\reg(\D)=n$. If $S\subseteq V$ is a subset such that $\widetilde{H}_{n-1}(\D[S])\neq 0$, we then have $\widetilde{H}_n(\IE(B_S(\D)))\neq 0$ by Corollary~\ref{cor:max-local} so that $\reg(B(\D))\geq n+1$.
\end{proof}

We remark that the gap between $\reg(\D)$ and $\reg(B(\D))$ could be arbitrarily large. For instance, if we define $F_n:=W(K_n)$ for some $n\geq 3$, where $W(K_n)$ is the graph obtained by attaching a pendant vertex to each vertex of $K_n$ (the \emph{whisker} of the complete graph), then $\reg(B(\Ind(F_n)))\geq n$, while $\reg(\Ind(F_n))=1$ since $F_n$ is cochordal.

The relation between a simplicial complex $\D$ and its bipartite graph $B(\D)$ can be reversible over the class of Sperner bipartite graphs.

\begin{proposition}\label{prop:bip-simp}
If $B$ is an $\spe$-bipartite graph on $X_i$, then the family
$$\D_i(B):=\{F\subseteq X_i\colon F\subseteq X_i\setminus N_{B}(v)\;\textrm{for\;some\;}v\in X_{1-i}\}$$
is a simplicial complex on $X_i$ for $i=0,1$. In particular,
we have $\reg(\D_i(B))+1\leq \reg(B)$.
\end{proposition}
\begin{proof}
The first claim is obvious. For the second, observe that the facets of
$\D_i(B)$ are of the form $X_i\setminus N_{B}(v)$ for $v\in X_{1-i}$; hence, we have $B(\D_i(B))\cong B$ so that the stated inequality follows from Theorem~\ref{cor:reg-bip-simp}. 
\end{proof}

Our next goal is to determine the induced matching number of $B(\D)$ from the combinatorics of the simplicial complex $\D$ that 
peculiarly involves the Helly number of $\D$.
\begin{definition}
Let $\D$ be a simplicial complex on $V$, and let $m\geq 3$ be an integer.
An $m$-\emph{flower} of $\D$ is a set $\{(x_1,F_1),\ldots,(x_m,F_m)\}$
of pairs $(x_i,F_i)\in V\times \FE_{\D}$ such that the set $\{x_1,\ldots,x_m\}$ is a face of $\D$, and $x_i\notin F_j$ if and only if $i=j$.
We say that $\D$ is \emph{flower-free}, if it has no $k$-flower for any $k\geq 3$. 
We then define 
\begin{equation}
\nu(\D):=\begin{cases}
0,& \text{if}\;\D\;\text{is flower-free},\\
\max \{k\in \N\colon \D\;\text{has a}\;k-\text{flower}\},& \text{otherwise,}
\end{cases}
\end{equation}
as the \emph{flower number} of $\D$.
\end{definition}

\begin{corollary}
$\im(B(\D))\geq \max \{2,\nu(\D)\}$
for any simplicial complex $\D$ other than a simplex.
\end{corollary}
\begin{proof}
Observe first that since $\D$ is not a simplex, it has at least two facets
$F_1,F_2$ together with $x\in F_1\setminus F_2$ and $y\in F_2\setminus F_1$. However, we then have $B(\D)[x,y,F_1,F_2]\cong 2K_2$, which forces $\im(B(\D))\geq 2$.

On the other hand, it is clear that any $m$-flower $\{(x_1,F_1),\ldots,(x_m,F_m)\}$ of $\D$ gives rise to an induced matching of $B(\D)$.
\end{proof}

Recall that a subset $K\subseteq V$ is said to be a minimal non-face of $\D$, if $K$ is itself not a face of $\D$ while any proper subset of it is. The maximum size of minimal non-faces of $\D$ is known as the \emph{Helly number} of $\D$, denoted by $\hf(\D)$~\cite{MTA}. This is consistent with the classical definition of the Helly number for set systems~\cite{DGG}. Indeed, if $\FE$ is a finite set system, then its Helly number can be defined by $\hf(\FE):=\hf(\NE(\FE))$, where $\NE(\FE)$ is the nerve complex of $\FE$. Observe that the inequality $\hf(\D)\leq \reg(\D)+1\leq \reg(B(\D))$ holds\footnote{The first part of this inequality also appears in~\cite{KM} for finite set systems.}, since if $K$ is a minimal non-face in $\D$ of size $\hf(\D)$, then $\D[K]\cong S^{\hf(\D)-2}$, where $S^m$ is the $m$-dimensional sphere.

\begin{corollary}\label{cor:im-bdelta}
$\im(B(\D))= \max \{2,\nu(\D), \hf(\D)\}$
for any simplicial complex $\D$ other than a simplex.
\end{corollary}
\begin{proof}
Suppose that $\im(B(\D))=m$, and let $\{(x_1,F_1),\ldots,(x_m,F_m)\}$ is
a maximum induced matching of size $m\geq 3$. If $K:=\{x_1,\ldots,x_m\}\in \D$, then $\{(x_1,F_1),\ldots,(x_m,F_m)\}$ is an $m$-flower so that 
$\nu(\D)\geq m$. On the other hand, if $K\notin \D$, then $K$ is a minimal non-face of $\D$ that implies $\hf(\D)\geq m$.
\end{proof}

One of the major open problem concerning our results in this section is the decision of whether any possible gap between $\reg(\D)$ and $\reg(B(\D))$ is caused because of a combinatorial or an algebraic property of the underlying simplicial complex $\D$. In other words,
we wonder whether there exists a combinatorially defined positive integer
$\eta(\D)$ such that the equality $\reg(B(\D))=\reg(\D)+\eta(\D)$ holds
for any simplicial complex $\D$ (over any field $\kk$). In the particular case of prime simplicial complexes, we predict that $\eta(\D)=1$.

\begin{problem}\label{prob:prime-eta}
Does the equality $\reg_{\kk}(\D)+1=\reg_{\kk}(B(\D))$ hold for any prime simplicial complex $\D$ over $\kk$?
\end{problem}


\subsection{Projective dimension and regularity of hypergraphs}
Once we have Theorem~\ref{cor:reg-bip-simp} for arbitrary simplicial complexes, we may naturally interpret it for the projective dimension and the regularity of hypergraphs. For that purpose, if we set $B(\HE):=B(\Ind(\HE))$, then Theorem~\ref{cor:reg-bip-simp} directly implies the following.

\begin{corollary}\label{cor:reg-hyper}
$\reg(\HE)+1\leq \reg(B(\HE))$ for any hypergraph $\HE$. 
\end{corollary}

We note that if $B$ is an $\spe$-bipartite graph on $X_i$, then the hypergraph
$\HE_i(B)$ on $X_i$ with edges $N_B(v)$ for $v\in X_{1-i}$ satisfies that $\LE(\HE_i(B))\cong B$ for $i=0,1$. This in particular implies that $\reg(\HE_i(B))+1\leq \reg(B)$ as a result of Corollary~\ref{cor:reg-hyper}.

We next show that the induced matching numbers of a hypergraph $\HE$ and that of $B(\HE)$ are in fact comparable. We recall first that if $\{F_1,\ldots,F_k\}$ is an induced matching of $\HE$, then the inequality $\sum_{i=1}^k(|F_i|-1)\leq \reg(\HE)$ holds~\cite{MV}. 

\begin{proposition}\label{prop:reg-bip-im-cd}
If $\{F_1,\ldots,F_k\}$ is an induced matching of $\HE$, then $\sum_{i=1}^k (|F_i|-1)\leq \im(B(\HE))$.
\end{proposition}
\begin{proof}
We let $F_i=\{x^i_1,\ldots,x_{k_i}^i\}$ for each $i\in [k]$, and define 
$$L_j^i:=\bigcup_{\underset{t\neq i}{t=1}}^k (F_t\setminus \{x_1^t\})\cup (F_i\setminus \{x^i_j\})$$
for any $i\in [k]$ and $2\leq j\leq k_i$. Since $\{F_1,\ldots,F_k\}$ is an induced matching, each set $L_j^i$ is an independent set in $\HE$. So, for each such set,
there exists a maximal independent set $S_j^i$ in $\HE$ such that $L_j^i\subseteq S_j^i$ and $x_j^i\notin S_j^i$. Now, the set 
$\{(x_j^i,S^i_j)\colon 1\leq i\leq k\;\text{and}\;2\leq j\leq k_i\}$ of edges 
forms an induced matching in $B(\HE)$ of size $\sum_{i=1}^k (|F_i|-1)$.
\end{proof}

\begin{remark}
We note that the equality is possible in Proposition~\ref{prop:reg-bip-im-cd}. For example, if we consider the graph $2K_2$, we have $B(2K_2)\cong C_8$ so that $\im(B(2K_2))=\im(2K_2)=2$.  
\end{remark}

When we employ the Terai's duality, Theorem~\ref{cor:reg-bip-simp} takes the following form:
\begin{corollary}\label{cor:pdlevi}
$\pd(\HE)\leq \reg(\LE(\HE))$ for any hypergraph $\HE$.
\end{corollary}
\begin{proof}
Since $\pd(\HE)=\reg(I^{\vee}_{\HE})=\reg(\Ind(\HE)^{\vee})+1$, where the first equality is the Terai's duality, the claim follows from Theorem~\ref{cor:reg-bip-simp} together with the fact that
$B(\Ind(\HE)^{\vee})\cong \LE(\HE)$, since any facet of $\Ind(\HE)^{\vee}$ is of the form $V\setminus F$, where $F\in \E$.
\end{proof}

We remark that the Helly number $\hf(\Ind(\HE)^{\vee})$ of the Alexander dual of $\Ind(\HE)$ exactly equals to $|V|-\ie(\HE)$, where $\ie(\HE)$ is the least cardinality of a maximal independent set of $\HE$, since any facet $F$ of $\IE(\HE)$ corresponds to a minimal non-face $V\setminus F$ of $\IE(\HE)^{\vee}$. This in particular implies that $\hf(\Ind(\HE)^{\vee})\leq \pd(\HE)$.

Our next target is to provide an upper bound on the regularity of the Levi graph $\LE(\HE)$ in terms of the upper independent vertex-wise domination number. However, we first recall the following result of H\'a and Woodroofe: 

\begin{theorem}\label{thm:ha-wood}~\cite{HW}
For any simplicial complex $\D$, we have $\reg(\D)$ to be at most the maximum size of a minimal face $A$ with the property that $\lk_{\D}(A)$ is a simplex.
\end{theorem}

Let $\D$ be a simplicial complex on $V$. We call a face $A\in \D$, an $S$-\emph{face} of $\D$ if $\lk_{\D}(A)$ is a simplex. Furthermore, an $S$-face $A$ is said to be \emph{minimal} if no proper subset of $A$ is an $S$-face of $\D$.

\begin{proposition}
An independent set $L\subseteq V$ is a minimal $S$-face of $\Ind(G)$ if and only if $L$ is a minimal vertex-wise dominating set for $G$.
\end{proposition}
\begin{proof}
Suppose first that $L$ is a minimal $S$-face of $\Ind(G)$, while it is not vertex-wise dominating for $G$. This means that there exists an edge $e=xy$ of $G$ such that $e$ is not vertex-wise dominated by $L$. However, this forces that $N_G[e]\cap L=\emptyset$, that is, $\{x\},\{y\}\in \lk_{\Ind(G)}(L)$ while $\{x,y\}\notin \lk_{\Ind(G)}(L)$, a contradiction.

Assume now that $L$ is a minimal vertex-wise dominating set of $G$. If $L$ is not an $S$-face of $\Ind(G)$, that is, $\lk_{\Ind(G)}(L)$ is not a simplex, then there exist two vertices $\{u\},\{v\}\in \lk_{\Ind(G)}(L)$ with $\{u,v\}\notin \lk_{\Ind(G)}(L)$. Note that this can only be possible when $f=uv\in E$. However, since $L$ is vertex-wise dominating for $G$, we must have $N_G[f]\cap L\neq \emptyset$ so that one of the vertices $x$ and $y$ can not be contained by the link of $L$ in $\Ind(G)$, a contradiction.
\end{proof}

\begin{corollary}\label{cor:betaev-s-face}
For any graph $G$ with $E\neq \emptyset$, we have 
\begin{equation*}
\beta(G)=\max\{|L|\colon L\;\textrm{is\;a\;minimal\;}S\textrm{-face\;of}\;\Ind(G)\}.
\end{equation*}
\end{corollary}

We are now ready to provide an upper bound on the regularity of any graph in terms of a domination parameter of the underlying graph that may be of independent interest.
\begin{theorem}\label{thm:reg-betaev}
$\reg(G)\leq \beta(G)$ for any graph $G$ with $E\neq \emptyset$.
\end{theorem}
\begin{proof}
This immediately follows from Theorem~\ref{thm:ha-wood} and Corollary~\ref{cor:betaev-s-face}.
\end{proof}

We remark that, in general, the invariants $\beta(G)$ and $\cd(G)$ are incomparable, where $\cd(G)$ is the \emph{cochordal cover number} of $G$, i.e., the least number of cochordal subgraphs $H_1,\ldots,H_k$ of $G$ satisfying $E(G)=\cup E(H_i)$. For example, $\cd(P_4)=1$ and $\beta(P_4)=2$, while $\beta(C_7)=2$ and $\cd(C_7)=3$. Furthermore, it is not difficult to construct examples of graphs showing that the gap between $\reg(G)$ and $\beta(G)$ could be arbitrarily large.

\begin{corollary}\label{cor:proddim-betaev}
$\pd(\HE)\leq \beta(\LE(\HE))$ for any hypergraph $\HE$ with $\E\neq \emptyset$.
\end{corollary}
\begin{proof}
Since $\pd(\HE)\leq \reg(\LE(\HE))$ for any hypergraph $\HE$ by Corollary~\ref{cor:pdlevi}, the claim follows from Theorem~\ref{thm:reg-betaev}.  
\end{proof}
\subsection{An application: Projective dimension of the dominance complex of graphs}

In this subsection, we compute the projective dimension of the dominance complex of any graph $G$, or equally that of the closed neighbourhood hypergraph of $G$ as an application of Corollary~\ref{cor:proddim-betaev}.

We first recall that for a given graph $G=(V,E)$, its \emph{dominance complex} $\dom(G)$ is the simplicial complex on $V$ whose faces are those subsets $A\subseteq V$ such that the set $V\setminus A$ is a dominating set for $G$~\cite{MT}. Observe that there exists a one-to-one correspondence between facets of $\dom(G)$ and minimal dominating sets of $G$. In order to compute the projective dimension of $\dom(G)$, we first view $\dom(G)$ as an independence complex of a hypergraph associated to $G$.

The \emph{closed neighbourhood hypergraph} $\NE[G]$ of a graph $G$ is defined to be the hypergraph on $V$ whose edges are the minimal elements of the poset
$\{N_G[x]\colon x\in V\}$ ordered with respect to the inclusion.

\begin{lemma}\label{lem:dom-inde}
$\dom(G)\cong \Ind(\NE[G])$ for any graph $G$.
\end{lemma}
\begin{proof}
Since any face of $\dom(G)$ is clearly an independent set of $\NE[G]$,
to prove the claim, it suffices to verify that when $X\subseteq V$ is an independent set of $\NE[G]$, the set $V\setminus X$ is dominating for $G$.
So, suppose that $x\in X$. It then follows from the definition that there exists $x'\in N_G[x]$ such that $N_G[x']\in \E(\NE[G])$. However, since $X$ is an independent set, we must have $N_G[x']\nsubseteq X$ so that there exists $v\in N_G[x']$ such that $v\notin X$. Now, since $N_G[x']\subseteq N_G[x]$, we conclude that $vx\in E$; hence, the vertex $x$ is dominated by
$v\in V\setminus X$. 
\end{proof}

\begin{proposition}\label{prop:beta-gamma}
$\Upsilon(\LE(\NE[G]))=\Gamma(G)$ for any graph $G$.
\end{proposition}
\begin{proof}
Suppose that $A$ is a vertex-wise dominating set of $\LE(\NE[G])$ for which we write $A=A_v\cup A_e$, where $A_v\subseteq V$ and $A_e\subseteq \E(\NE[G])$. Observe that if $F\in A_e$, then there exists $x\in V$ such that $F=N_G[x]$. Now, we define $B:=\{x\in V\colon N_G[x]\in A_e\}$, and claim that the set $C:=A_v\cup B$ is a minimal dominating set for $G$. Let $v\in V\setminus C$ be given. It then follows from the definition that there exists $v'\in N_G[v]$ such that
$N_G[v']\in \E(\NE[G])$. If $v'\in C$, then $v$ is dominated by the vertex $v'$ so that we may further assume that $v'\notin C$. Now, consider the edge $(v, N_G[v'])$ of $\LE(\NE[G])$. Since $A$ is vertex-wise dominating, the edge $(v, N_G[v'])$ must be vertex-wise dominated by a vertex in $A$. If $u\in A_v$ is a vertex that vertex-wise dominates the edge $(v, N_G[v'])$, then $u\in N_G[v']\subseteq N_G[v]$ so that $uv\in E$, that is, $u$ dominates $v$ in $G$. On the other hand, if there exists a vertex $w\in V$ such that $N_G[w]\in A_e$ vertex-wise dominates the edge $(v, N_G[v'])$, then we have $v\in N_G[w]$. However, it means that the vertex $w\in B$ dominates $v$ in $G$. This proves that $C$ is a dominating set for $G$.  

For the other direction, suppose that $S\subseteq V$ is a dominating set for $G$. We claim that it is also a vertex-wise dominating set for $\LE(\NE[G])$. Consider an edge $(x, N_G[y])$ of $\LE(\NE[G])$. Since $S$ is a dominating set for $G$, we must have $N_G[y]\cap S\neq \emptyset$.
This in particular forces that there exists $s\in S$ such that $sy\in E$. However, it then follows that the vertex $s$ vertex-wise dominates the edge $(x, N_G[y])$.
\end{proof}
\begin{theorem}
$\pd(\dom(G))=\Gamma(G)$ for any graph $G$.
\end{theorem}
\begin{proof}
We clearly have $\ie(\NE[G])=|G|-\Gamma(G)$ so that
$\hf(\dom(G))=|G|-\ie(\NE[G])=\Gamma(G)\leq \pd(\dom(G))$. On the other hand, we have 
\begin{align*}
\pd(\dom(G))=\pd(\NE[G])&\leq \reg(\LE(\NE[G]))\\
&\leq \beta(\LE(\NE[G]))\leq \Upsilon(\LE(\NE[G]))=\Gamma(G)
\end{align*} 
by Proposition~\ref{prop:beta-gamma}.  
\end{proof}

\section{Projective dimension of graphs and regularity of subdivisions}

In this section, we provide a detail analysis on the consequences of Corollary~\ref{cor:pdlevi} when $\HE=H$ is just a  graph. This will include the comparison of upper and lower bounds both on projective dimension of a graph $G$ and those of the regularity of its subdivision graph $S(G)$. However, we first restate  Corollary~\ref{cor:pdlevi} in the graph's case for completeness:

\begin{corollary}\label{cor:prodtosubreg}
$\pd(G)\leq \reg(S(G))$ for any graph $G$.
\end{corollary}

\begin{proposition}
For any graph $G$, there exists a subgraph $H$ (not necessarily induced) of $G$ such that $\pd(H)=\reg(S(G))$.
\end{proposition}
\begin{proof}
Suppose that $\reg(S(G))=m$, and let $R\subseteq V$ and $F\subseteq E$ be subsets satisfying $\widetilde{H}_{m-1}(S(G)[R\cup F])\neq 0$ for which we may assume that both sets $R$ and $F$ are minimal with this property. Now, we define $H:=(V(F),F)$ and claim that it satisfies the required property. In order to prove that it suffices to verify the homotopy equivalence
$S(H)\simeq S(G)[R\cup F]$, since we already have $\Sigma(\Ind(H)^{\vee})\simeq S(H)$.
If $x\in V(F)\setminus R$ and $x\in N_{S(G)[R\cup F]}(f)$ for some $f\in F$, 
then $\deg_{S(G)[R\cup F]}(f)=1$ by the minimality of the sets $R$ and $F$; hence, if $y$ is the only neighbour of $f$ in $S(G)[R\cup F]$, then we must have $N_{S(H)}(y)\subseteq N_{S(H)}(x)$, which in turn implies that
$S(H)\simeq S(H)-x$. It then follows that we can remove any vertex in $V(F)\setminus R$ without altering the homotopy type of $S(H)$, that is,
$S(H)\simeq S(H)-(V(F)\setminus R)\cong S(G)[R\cup F]$.
\end{proof}

\begin{lemma}\label{lem:sg-reduct}
Let $G=(V,E)$ be a graph with $\reg(S(G))=m>0$. Then for any $m>k\geq 0$,
there exists an induced subgraph $H_k$ of $G$ such that
$\reg(S(H_k))=k$.
\end{lemma}
\begin{proof}
By the prime factorization theorem of~\cite{BC1}, the graph $S(G)$ has a prime factorization, say $\{R_1,\ldots,R_s\}$. Now, if $v\in V(R_1)\cap V$,
we have $\reg(R_1)=\reg(R_1-N_{R_1}[v])+1$, since $R_1$ is a prime graph.
However, since the disjoint union $(R_1-N_{R_1}[v])\cup R_2\cup \ldots \cup R_s$ is an induced subgraph of $S(G)-N_{S(G)}[v]\cong S(G-v)$; it then follows that
\begin{align*}
\reg(S(G))\geq \reg(S(G-v))\geq& \reg(R_1-N_{R_1}[v])+\reg(R_2)+\ldots+\reg(R_s)\\
=&(\reg(R_1)-1)+\reg(R_2)+\ldots+\reg(R_s)\\
=&\reg(S(G))-1.
\end{align*}
Therefore, we have either $\reg(S(G))=\reg(S(G-v))$ or else $\reg(S(G))=\reg(S(G-v))+1$ so that the claim follows from the induction on the order of $G$.
\end{proof}
One of the immediate consequence of Lemma~\ref{lem:sg-reduct} is the following:
\begin{corollary}\label{cor:ind-pd-reg}
Any graph $G$ has an induced subgraph $H$ such that $\pd(G)=\reg(S(H))$.
\end{corollary}

We note that even if Corollary~\ref{cor:ind-pd-reg} does not exactly determines the corresponding induced subgraph $H$ of $G$ satisfying
$\pd(G)=\reg(S(H))$, it naturally translates the projective dimension of graphs to that of the regularity of bipartite graphs.

Our next aim is to compare various upper and lower bounds on the projective dimension and the regularity of graphs. We begin with interpreting the relevant invariants of a graph $G$ in terms of those of the subdivision graph $S(G)$. 

\begin{proposition}\label{prop:imsg}
$\im(S(G))=|G|-\gamma(G)$ for any graph $G$ without any isolated vertices.
\end{proposition}
\begin{proof}
Suppose that $\gamma(G)=|D|$ for some dominating set $D\subseteq V$. Then for any vertex $x\in V\setminus D$, there exists a vertex $d$ such that $xd\in E$. Now, the set $\{(x,xd)\colon x\in V\setminus D\}$ of edges of $S(G)$ forms an induced matching in $S(G)$; hence, we have $\im(S(G))\geq |G|-\gamma(G)$. On the other hand, let $M=\{(x_1,e_1),\ldots,(x_n,e_n)\}$ be a maximum induced matching of $S(G)$. We claim that the set $D:=V\setminus \{x_1,\ldots,x_n\}$ is a dominating set for $G$ so that $\im(S(G))\leq |G|-\gamma(G)$. Indeed, for any vertex $x_i\in V\setminus D$, the other end vertex of $e_i$, say $y_i$, must be contained in $D$, since otherwise the set $M$ would not be an induced matching; hence, we have 
$x_i\in N_G(D)$.
\end{proof}

We recall that the equality $\hf(\IE(G)^{\vee})=|G|-\ie(G)$ holds for any graph $G$ by our previous remark for hypergraphs. On the other hand, the flower number $\nu(S(G)):=\nu(\Ind(G)^{\vee})$ is rather complicated to describe. 

We call a subset $X\subseteq V$ as an \emph{unstable dominating set} of $G$, if $X$ is a dominating set for $G$ with $E(G[X])\neq \emptyset$. The least cardinality of an unstable dominating set is called the \emph{unstable domination number} of $G$ and denoted by $\gamma_{\us}(G)$.
Observe that the inequality $\gamma(G)\leq \gamma_{\us}(G)\leq \gamma(G)+1$ always holds for any graph $G$.

\begin{lemma}
$\nu(S(G))=|G|-\gamma_{\us}(G)$ for any graph $G$ without any isolated vertices.
\end{lemma}
\begin{proof}
Suppose that $\{(x_1,e_1),\ldots,(x_n,e_n)\}$ is a maximum flower in $S(G)$. This in particular implies that the set $X:=\{x_1,\ldots,x_n\}$ is face of $\Ind(G)^{\vee}$. However, this means that $X$ misses at least one edge of $G$. It then follows that $V\setminus X$ is an unstable dominating set for $G$, since each vertex $x_i$ is dominated by the other end vertex
of $e_i$ contained in $V\setminus X$; hence, $\nu(S(G))\leq |G|-\gamma_{\us}(G)$.

Assume next that $Y\subseteq V$ is an unstable dominating set for $G$ with
$\gamma_{\us}(G)=|Y|$. Then, for each $u\in V\setminus Y$, there exists a vertex $y\in Y$ such that $uy\in E$. However, the set $\{(u,uy)\colon u\in V\setminus Y\}$ forms a flower for $S(G)$, since $E(G[Y])\neq \emptyset$; hence, we conclude that $|G|-\gamma_{\us}(G)\leq \nu(S(G))$.
\end{proof}

We note that the equality $\max\{|G|-\gamma_{\us}(G), |G|-\ie(G)\}=|G|-\gamma(G)$ clearly holds for any graph $G$ so that Proposition~\ref{prop:imsg} also follows from Corollary~\ref{cor:im-bdelta}.

\begin{corollary}\label{cor:im-prod}
If $\gamma(G)=\ie(G)$, then $\im(S(G))\leq \pd(G)\leq \reg(S(G))$.
\end{corollary}
\subsection{Domination bounds and the regularity of subdivisions}
We prove in this subsection that all known upper bounds to projective dimension of graphs are in fact upper bounds to the regularity of  subdivision graphs. Before we proceed further, we state the following upper and lower bounds on the projective dimension of graphs due to Dao and Schweig~\cite{DS1}.
\begin{theorem}\label{thm:ds1}~\cite{DS1}
For any graph $G$,
\begin{equation*}
|G|-\ie(G)\leq \pd(G)\leq |G|-\max\{\tau(G),\epsilon(G)\}.
\end{equation*}
\end{theorem}
We next verify that both upper bounds of Theorem~\ref{thm:ds1} are in fact upper bounds to $\reg(S(G))$. For that purpose, we reformulate the related graph parameters of $G$ in terms of those of the graph $S(G)$.

We recall that for any graph $G=(V,E)$, the \emph{square} $G^2$ of $G$ is defined to be the graph on $V(G^2)=V$ such that $xy\in E(G^2)$ if and only if $\dis_G(x,y)\leq 2$. Moreover, we denote by $\alpha(G)$, the \emph{independence number} of the graph $G$.

\begin{lemma}\label{lemma:tau}
For any graph $G$, the equality $\tau(G)=\alpha(G^2)$ holds.
\end{lemma}
\begin{proof}
Assume that $\alpha(G^2)=k$, and let $A=\{a_1,\ldots,a_k\}$ be an independent set of $G^2$. Note that $\dis_G(a_i,a_j)\geq 3$ for any $i\neq j$, and in particular, the set $A$ is also an independent set of $G$.
If $X\subseteq V$ is a subset satisfying $A\subseteq N_G(X)$, we necessarily have  $|A|\leq |X|$, since $|A\cap N_G(x)|\leq 1$ for any $x\in X$. However, it then follows that $k\leq |X|\leq \tau(G)$, that is, $\alpha(G^2)\leq \tau(G)$.

Suppose now that $\tau(G)=n$, and let $A\subseteq V$ be an independent subset of $G$ of minimal order satisfying $\gamma(A;G)=n$. Furthermore, we
let $X\subseteq V$ be a subset such that $A\subseteq N_G(X)$ and $|X|=n$.
Since $X$ is minimal, each vertex in $X$ has at least one private neighbour in $A$. In other words, we have $P_A^X(x)\neq \emptyset$ for any $x\in X$.

On the other hand, since $A$ is minimal, every vertex of $A$ must be the private neighbour of a vertex in $X$. In other words, for each $a\in A$, there exists a unique $x_a\in X$ such that $a\in P_A^X(x_a)$. We then claim that $A$ is an independent set of $G^2$. Assume otherwise that there exist $a,b\in A$ such that $\dis_G(a,b)=2$, and let $u\in N_G(a)\cap N_G(b)$ be any vertex. Observe that $u\notin X$  by the minimality of $A$. However, it then follows that $A\subseteq N_G(Y)$, where 
$Y:=(X\setminus \{x_a,x_b\})\cup \{u\}$, a contradiction. Therefore, we conclude that $\alpha(G^2)\geq |A|\geq |X|=n$. This completes the proof.
\end{proof}

The following result of Horton and Kilakos~\cite{HK} relates the independence number of the square of a graph to the edge domination number of the subdivision graph.

\begin{theorem}\label{thm:hk}~\cite{HK}
If $G$ is a graph, then $\gamma'(S(G))+\alpha(G^2)=|G|$.
\end{theorem}

We now arrive the proof of the first part of our claim regarding the upper bounds of Theorem~\ref{thm:ds1}. 

\begin{corollary}\label{cor:cochord-tau}
The equality $\cd(S(G))=|G|-\tau(G)$ holds for any graph $G$. In particular, we have $\pd(G)\leq \cd(S(G))$.
\end{corollary}
\begin{proof}
We proved in~\cite{BC} that if $G$ is a graph with $\gi(G)\geq 5$, then $\cd(G)=\gamma'(G)$. Since $\gi(S(G))\geq 6$, the result follows from Lemma~\ref{lemma:tau} and Theorem~\ref{thm:hk}. The last claim is the consequence of Corollary~\ref{cor:prodtosubreg} together with the fact that
$\reg(H)\leq \cd(H)$ for any graph $H$.
\end{proof}

\begin{corollary}
If $G$ is a chordal graph, then $\reg(S(G))=\im(S(G))=|G|-\gamma(G)$. 
\end{corollary}
\begin{proof}
The equality $\gamma(G)=\tau(G)$ holds for any chordal graph~\cite{AS}; hence, the claim follows from Proposition~\ref{prop:imsg} and Corollary~\ref{cor:cochord-tau}.
\end{proof}


We next verify that the edge-wise domination number of $G$ is closely related to the upper vertex-wise domination number of the subdivision graph of $G$.

\begin{theorem}\label{thm:sub-ve}
$\Upsilon(S(G)))=|G|-\epsilon(G)$ for any graph $G$ without any isolated vertex.
\end{theorem}
\begin{proof}
Let $M$ be an edge-wise dominating matching of $G$ with $|M|=\epsilon(G)$, and let $U:=U(M)$ be the set of unmatched vertices of $G$ with respect to $M$. We then claim that the set $U\cup M$ is a minimal vertex-wise dominating set for $S(G)$. Observe that it only suffices to prove that $U\cup M$ is minimal, since it is clearly a vertex-wise dominating set. 
If $u\in U$, then there exists a $v\in V(M)$ such that $e_u=uv\in E$, since $M$ is an edge-wise dominating matching of $G$. In particular, the edge $(u,e_u)$ of $S(G)$ is only dominated by the vertex $u$, that is, the edge $(u,e_u)$ is a vertex-wise private neighbour of $u$. On the other hand, if $f=xy\in M$, then the edge $(x,f)$ of $S(G)$ is a vertex-wise private neighbour of $f$; hence, the claim follows. However, we then have
$|G|-\epsilon(G)=|G|-|M|=|U|+|M|\leq \Upsilon(S(G))$.

We prove the reversed inequality in three steps. Note that when $A$ is a vertex-wise dominating set of $S(G)$, we write $A=A_v\cup A_e$ such that $A_v\subseteq V$ and $A_e\subseteq E$. We first verify that there exists a vertex-wise dominating set $A$ of $S(G)$ of appropriate size such that $A_e\neq \emptyset$. We then show that among any such vertex-wise dominating sets, we can find one $A$ for which the set $A_e$ is an edge-wise dominating matching of $G$.

{\it Claim $1$.} For any connected graph $G$, there exists a minimal vertex-wise dominating set $A$ with $\Upsilon(S(G))=|A|$ such that $A_e\neq \emptyset$.

{\it Proof of Claim $1$:} Suppose that $S(G)$ admits a minimal vertex-wise dominating set $A$ with $\Upsilon(S(G))=|A|$ such that $A_e=\emptyset$, that is, $A=A_v\subseteq V$. We note that for each $x\in A$, there exists a vertex $y_x\in N_G(x)$ such that $y_x\notin A$, since otherwise $A$ could not be minimal. We claim that such a vertex $y_x\in N_G(x)$ is unique for each $x\in A$. Suppose otherwise that there exist $x\in A$ together with $y_1,\ldots, y_k\in N_G(x)$ such that $y_1,\ldots,y_k\notin A$.
However, it then follows that the set $(A\setminus \{x\})\cup \{e_1,\ldots, e_k\}$, where $e_i=y_ix$ for $i\in [k]$, would be a minimal vertex-wise dominating set for $S(G)$ having larger size than $A$, which is not possible. Therefore, the vertex $y_x$ must be unique for each $x\in A$. Now, if we let $B:=(A\setminus\{x\})\cup \{e_x\}$, where $e_x:=y_xx\in E$, then the set $B$ is a minimal vertex-wise dominating set of $S(G)$ with $B_e\neq \emptyset$ and $|B|=|A|=\Upsilon(S(G))$.

{\it Claim $2$.} For any connected graph $G$, there exists a minimal vertex-wise dominating set $A$ with $\Upsilon(S(G))=|A|$ such that $A_e$ is an edge-wise dominating set for $G$.

{\it Proof of Claim $2$:} We first verify that such a set $A_e$ can be turned into an edge-wise dominating set for $G$. Suppose that $A_e$ is not edge-wise dominating for $G$. So, there exists a vertex $x\in V$ such that $x$ is neither an end vertex of an edge in $A_e$ nor it is adjacent to any end vertex of an edge in $A_e$. Assume first that $x\notin A_v$. This in particular implies that $N_G(x)\subseteq A_v$. Pick a neighbour, say $y\in N_G(x)$. Observe that if $f$ is an edge of $G$ incident to $y$, then $f\notin A_e$. If $N_G(y)\setminus \{x\}\subseteq A_v$, then the set $(A\setminus \{y\})\cup \{xy\}$ is a minimal vertex-wise dominating set of $S(G)$ having size $\Upsilon(S(G))$ such that $x$ is dominated by the edge $xy$ in $G$. Thus, we may assume that the vertex $y$ has at least one neighbour in $G$ other than $x$ not contained by $A_v$. We denote by $L_y$, the set of such vertices. We note that the set $L_y$ can not contain a vertex $w$ such that any edge of $G$ incident to $w$ is not contained by $A_e$. Indeed, if we add the edge $yw$ to $(A\setminus \{y\})\cup \{xy\}$ for any such vertex $w\in L_y$, the  resulting set would be a minimal vertex-wise dominating set for $S(G)$, which is not possible. This forces that every vertex $w$ in $L_y$ has at least one incident edge in $A_e$. However, it then follows that the set 
$(A\setminus \{y\})\cup \{xy\}$ is a minimal vertex-wise dominating set of
$S(G)$ for which the vertex $x$ is now dominated by the edge $xy$ in $G$.

Assume next that $x\in A_v$. In such a case there must exist at least one vertex in $N_G(x)$ not contained in $A_v$ by the minimality of $A$. We denote by $T_x$, the set of such vertices. Observe that no incident edges to a vertex $z$ in $T_x$ can be contained in $A_e$ by the choice of $x$. However, this forces that any vertex adjacent to a vertex in $T_x$ must belong to $A_v$. Now, if we define
$A':=(A\setminus \{x\})\cup \{xz\colon z\in T_x\}$, then the set $A'$ is clearly a minimal vertex-wise dominating set for $S(G)$. However, since $|A|=\Upsilon(S(G))$, this could be only possible when $|T_x|=1$. If $T_x=\{z\}$, then the set $(A\setminus \{x\})\cup \{xz\}$
is a minimal vertex-wise dominating set for $S(G)$ for which the vertex $x$ is dominated by the edge $xz$ in $G$.

As a result, since we only replace a vertex in $V$ by an edge in $E$ for $A$, the resulting set $A_e$ is always an edge-wise dominating set of $G$.

{\it Claim $3$.} For any connected graph $G$, there exists a minimal vertex-wise dominating set $A$ with $\Upsilon(S(G))=|A|$ such that $A_e$ is an  edge-wise dominating matching for $G$.

{\it Proof of Claim $3$:} Suppose that $A$ is a vertex-wise dominating set with $\Upsilon(S(G))=|A|$ satisfying the following properties:
\begin{itemize}
\item[$(i)$] $A_e$ is an edge-wise dominating set for $G$ of minimum possible size among any minimal vertex-wise dominating sets for $S(G)$ of maximum size,\\
\item[$(ii)$] the set $A_e$ contains least number of pairs of incident edges in $G$ among sets satisfying the condition $(i)$.
\end{itemize}
Let $e=xy$ and $f=yz$ be two such edges in $A_e$. This in particular implies that neither $x$ nor $z$ can belong to $A_v$. Moreover, no edge incident to either $x$ or $z$ can be contained in $A_e$ by the minimality of $A$. Note also that we must have $\deg_G(x), \deg_G(z)>1$. For instance, if $\deg_G(x)=1$, then the set $B:=(A\setminus \{e\})\cup \{x\}$ is a minimal vertex-wise dominating set for $S(G)$ in which $B_e$ has fewer pairs of adjacent edges than $A_e$, a contradiction. So, let $w$ be a neighbour of $x$ in $G$. If $w\notin A_v$, we then define $C:=(A\setminus \{e\})\cup \{xw\}$ so that the set $C$ provides a minimal vertex-wise dominating set for $S(G)$ in which $C_e$ has fewer pairs of incident edges than $A_e$. Therefore, we must have $N_G(x)\setminus \{y\}\subseteq A_v$. We claim that not every vertex in $N_G(x)\setminus \{y\}$ can be edge-wise dominated by some edge in $A_e$. Since otherwise, the set $B:=(A\setminus \{e\})\cup \{x\}$ would be a minimal vertex-wise dominating set for $S(G)$ for which $|B_e|<|A_e|$, violating the condition $(i)$. So, we define $K_x$ to be the subset of
$N_G(x)\setminus \{y\}$ consisting of those vertices being the edge-wise private neighbours of the edge $e=xy$. If we pick a vertex, say $p\in K_x$, and define $D:=(A\setminus \{e\})\cup \{xp\}$, then the set $D$ is a minimal vertex-wise dominating set for $S(G)$ with $|D|=|A|$. In particular, $D_e$ is an edge-wise dominating set for $G$ having fewer pairs of incident edges than $A_e$, a contradiction. This completes the proof of Claim $3$.

Finally, we conclude that $\Upsilon(S(G))=|A|=|A_v|+|A_e|\leq |G|-|A_e|$, where the inequality is due to the fact that $A_e$ is a matching for $G$. On the other hand, since $A_e$ is an edge-wise dominating set for $G$, we must have
$|A_e|\geq \epsilon(G)$ from which we conclude that $\Upsilon(S(G))\leq |G|-\epsilon(G)$.
\end{proof}

The first part of the proof of Theorem~\ref{thm:sub-ve} together with the fact that $\beta(H)\leq \Upsilon(H)$ for any graph $H$ give rise to the following fact.
\begin{corollary}
$\beta(S(G))=\Upsilon(S(G))=|G|-\epsilon(G)$ for any graph $G$.
\end{corollary}

We remark that it is not known in general whether the equality 
$\beta(H)=\Upsilon(H)$ holds for any graph $H$~\cite{JL}. 

The following result completes the proof of our claim regarding the upper bounds of Theorem~\ref{thm:ds1}:
\begin{corollary}\label{cor:sg-ev}
$\reg(S(G))\leq \Upsilon(S(G))=|G|-\epsilon(G)$ for any graph $G$.
\end{corollary}
\begin{proof}
The claim follows from Theorems~\ref{thm:reg-betaev} and~\ref{thm:sub-ve}.  
\end{proof}

We close this section by showing that the regularity of any bipartite graph can be bounded from above by its induced matching number from which we may readily deduced an upper bound on the projective dimension of graphs.

\begin{theorem}\label{thm:reg-im-sg}
If $B$ is a bipartite graph with a bipartition $V(B)=X\cup Y$, then $\reg(B)\leq \frac{1}{2}(\im(B)+\min\{|X|,|Y|\})$.
\end{theorem}
\begin{proof}
We only prove the inequality $\reg(B)\leq \frac{1}{2}(\im(B)+|Y|)$ from which the general claim follows.

We start with the graph $B_0:=B$ having the bipartition $X_0:=X$ and $Y_0:=Y$. We then introduce a reduction process that creates a bipartite  graph $B_i$ in each state $i\geq 0$, and associate a pair of integers
$(\prc_i,\imc_i)$ to $B_i$ by letting $\prc_0:=0$ and $\imc_0:=0$ at the bottom. 

Now, pick a vertex $x_i\in X_i$ for $i\geq 0$, and assume first that $x_i$ is not a prime vertex of $B_i$. If $\deg_{B_i}(x_i)\geq 2$, we define  $B_{i+1}$ to be the bipartite graph having the bipartition $V(B_{i+1})=X_{i+1}\cup Y_{i+1}$, where $X_{i+1}:=X_i-x_i$, $Y_{i+1}:=Y_i$ together with $(\prc_{i+1},\imc_{i+1})=(\prc_i,\imc_i)$. 

On the other hand, if $\deg_{B_i}(x_i)=1$ and $y_i$ is the only neighbour of it in $B_i$, we let
$B_{i+1}$ to be the bipartite graph with $X_{i+1}:=X_i-N_{B_i}(y_i)$, $Y_{i+1}:=Y_i-y_i$, and set $\imc_{i+1}:=\imc_i+1$, while $\prc_{i+1}:=\prc_i$ (compare to Lemma~$3.25$ of~\cite{BC1}). 

If $x_i$ is a prime vertex of $B_i$, we take $B_{i+1}$ to be the bipartite graph with $X_{i+1}:=X_i-x_i$, $Y_{i+1}:=Y_i-N_{B_i}(x_i)$, and set $\prc_{i+1}:=\prc_i+1$ and $\imc_{i+1}:=\imc_i$. We proceed this process until there exists a $j\geq 0$ such that $X_j=\emptyset$ or $Y_j=\emptyset$. 

Observe that $\reg(B)=\prc_j+\imc_j$ and that the inequalities
$\imc_j\leq \im(B)$ and $\prc_j\leq \frac{1}{2}(|Y|-\imc_j)$
hold. However, we then have 
\begin{align*}
\reg(B)=\prc_j+\imc_j&\leq \frac{1}{2}(|Y|-\imc_j)+\imc_j\\
&=\frac{1}{2}(|Y|+\imc_j)\\
&\leq \frac{1}{2}(|Y|+\im(B)).
\end{align*}
\end{proof}

\begin{corollary}\label{cor:prod-dim-gamma}
For any graph $G$, we have $\pd(G)\leq |G|-\frac{1}{2}\gamma(G)$.
\end{corollary}
\begin{proof}
Since $\im(S(G))=|G|-\gamma(G)$ by Proposition~\ref{prop:imsg}, it then follows from Corollary~\ref{cor:prodtosubreg} together with Theorem~\ref{thm:reg-im-sg} that
$$\pd(G)\leq \reg(S(G))\leq \frac{1}{2}(\im(S(G))+|G|)=\frac{1}{2}(|G|-\gamma(G)+|G|)=|G|-\frac{1}{2}\gamma(G).$$
\end{proof}

Observe that the bound of Corollary~\ref{cor:prod-dim-gamma} is sharp as the graph $C_4$ satisfies that $\pd(C_4)=3=\reg(C_8)=|C_4|-\frac{1}{2}\gamma(C_4)=3$.

We leave it open the discussion whether the hypergraph versions of the domination parameters, the independence domination number $\tau(\HE)$ and the edge-wise domination number $\epsilon(\HE)$ of a hypergraph $\HE$ as they are introduced in~\cite{DS2,DS3} provide upper bounds on the regularity of the Levi graph of $\HE$.
However, we prefer to state an interesting connection between the parameter $\epsilon(\HE)$ and the induced matching number of the corresponding Levi graph.
\begin{definition}
Let $B$ be a bipartite graph with a bipartition $X\cup Y$ containing no isolated vertex. We define $h_X(B)$ to be the least cardinality of a subset $S\subseteq X$ such that for every $y\in Y$ there exists an $s\in S$ satisfying $\dis_B(s,y)\leq 3$ (the parameter $h_Y(B)$ can be defined analogously).
\end{definition}
\begin{lemma}
$\epsilon(\HE)=h_{\E}(\LE(\HE))\leq \im(\LE(\HE))$ for any hypergraph $\HE=(V,\E)$.
\end{lemma}
\begin{proof}
The equality is a direct consequence of the definitions, and for the inequality, suppose that $\{E_1,\ldots,E_k\}\subseteq \E$ is a minimal edge-wise dominating set of cardinality $\epsilon(\HE)$. By the minimality, there exists a vertex $v_i\in E_i$ which is privately edge-wise dominated by $E_i$ in $\LE(\HE)$. However, the set $\{(v_1,E_1),\ldots,(v_k,E_k)\}$ forms an induced matching in $\LE(\HE)$. 
\end{proof}
Observe that Theorem~\ref{thm:reg-im-sg} naturally implies that
$\reg(\LE(\HE))\leq |\HE|-h_{\E}(\HE)$ whenever $\im(\LE(\HE))\leq |\HE|/3$ or $\reg(\LE(\HE))\leq |\HE|/2$. On the other hand, the inequality $\tau(\HE)+h_V(\LE(\HE))\leq |\HE|$ trivially holds for any hypergraph $\HE$. In particular, we wonder whether such a special notion of a distance-three domination on bipartite graphs always provides an upper bound to the regularity of such graphs.
\begin{question}
Does the inequality $\reg(B)\leq \max\{|X|-h_Y(B), |Y|-h_X(B)\}$ hold for any bipartite graph with a bipartition $X\cup Y$ containing no isolated vertex?
\end{question}
\section{Projectively prime graphs}
In this section, in analogy with the notion of prime graphs introduced for the regularity calculations of graphs~\cite{BC1}, we demonstrate that a similar notion could be useful for the projective dimension of graphs as well. 

\begin{definition}
A connected graph $G$ is called a \emph{projectively prime graph} over a field $\kk$, if $\pd_{\kk}(G-x)<\pd_{\kk}(G)$ for any vertex $x\in V(G)$. 
Furthermore, we call a connected graph $G$ as a \emph{perfect projectively prime graph} if it is a projectively prime graph over any field. 
\end{definition}

Obvious examples of perfect projectively prime graphs can be given by the paths $P_{3k}$ and $P_{3k+2}$ for any $k\geq 1$, and cycles $C_n$ for any $n\geq 3$. The primeness of the corresponding paths directly follows from Corollary~\ref{cor:im-prod} together with the fact that $\reg(S(P_m))=\im(S(P_m))=\lfloor \frac{2m-1}{3}\rfloor$ for any $m\geq 2$.
Apart from these, our next result shows that the join of any two graphs is a perfect projectively prime graph. 

When $G$ and $H$ are two graphs on the disjoint set of vertices, we write $G*H$, the \emph{join} of $G$ and $H$, for the graph obtained from the disjoint union $G\cup H$ by inserting an edge between any two vertices $x\in V(G)$ and $y\in V(H)$. The calculation of projective dimension
of the graph $G*H$ seems to first appear in~\cite{AM}, while we here provide an independent proof.

\begin{proposition}\label{prop:join-prime}
If $G$ and $H$ are two graphs on the disjoint sets of vertices, then the graph $G*H$ is a perfect \bibliography{prod-dim}
projectively prime graph. In particular, we have $\pd(G*H)=\reg(S(G*H))=|G|+|H|-1$.
\end{proposition}
\begin{proof}
We prove the second claim from which the projectively primeness of $G*H$ directly follows.
Since $\epsilon(G*H)=1$, we have $\pd(G*H)\leq \reg(S(G*H))\leq |G|+|H|-1$ by Corollary~\ref{cor:sg-ev}. On the other hand, the Alexander dual of $\Ind(G*H)$ triangulates a $(|G|+|H|-3)$-dimensional sphere. Indeed, the complex $\Ind(G*H)^{\vee}$ is just the simplicial join of the boundary complexes of corresponding simplexes on $V(G)$ and $V(H)$. Therefore, we have $|G|+|H|-1\leq \pd(G*H)$ from which we conclude the claim.
\end{proof}

We note that the complete multipartite graph $K_{n_1,\ldots,n_m}$ for any $n_i\geq 1$ and $m\geq 2$ is a perfect projectively prime graph as a result of Proposition~\ref{prop:join-prime}.

As in the case of prime graphs, the notion of projectively prime graphs allows us to reformulate the projective dimension as a generalized (weighted) induced matching problem that we describe next.

\begin{definition}
Let $G$ be a graph and let $\RE=\{R_1,\ldots, R_r\}$ be a set of pairwise vertex disjoint induced subgraphs of $G$ such that $|V(R_i)|\geq 2$ for each
$1\leq i\leq r$. Then $\RE$ is said to be an \emph{induced decomposition} of $G$ if the induced subgraph of $G$ on $\bigcup_{i=1}^r V(R_i)$
contains no edge of $G$ that is not contained in any of $E(R_i)$, and
and  $\RE$ is maximal with this property. The set of induced decompositions of a graph $G$ is denoted by $\ID(G)$. 
\end{definition}

\begin{definition}
Let $\RE=\{R_1,\ldots, R_r\}$ be an induced decomposition of a graph $G$. If each $R_i$ is a projectively prime graph, 
then we call $\RE$ as a \emph{projectively prime decomposition} of $G$, and the set of projectively prime decompositions of a graph $G$ is denoted by
$\PD(G)$. 
\end{definition}

Obviously, the set $\PD(G)$ is non-empty for any graph $G$.

\begin{theorem}\label{thm:pd-prime-factor}
For any graph $G$ and any field $\kk$, we have $$\pd_{\kk}(G)=\max \{\sum_{i=1}^{r}\pd_{\kk}(H_i)\colon \{H_1,\ldots,H_r\}\in \PD_{\kk}(G)\}.$$
\end{theorem}
\begin{proof}
If $G$ is itself a projectively prime graph, there is nothing to prove. Otherwise there exists a vertex $x\in V$ such that $\pd(G)=\pd(G-x)$. If $G-x$ is a projectively prime graph, then $\{G-x\}\in \PD(G)$ so that the result follows. Otherwise, we have 
$\pd(G-x)=\max \{\sum_{i=1}^{t}\pd(S_i)\colon \{S_1,\ldots,S_t\}\in \PD(G-x)\}$ by the induction. However,
since $\PD(G-x)\subseteq \PD(G)$ for such a vertex, the claim follows.
\end{proof}

By taking the advantage of Theorem~\ref{thm:pd-prime-factor}, we next present examples of graphs showing that most of the bounds on the projective dimension of graphs involving domination parameters are in fact far from being tight. Beforehand, we need the following well-known fact first.

\begin{lemma}~\cite{DHS}
For any vertex $x$ of a graph $G$,
\begin{equation}\label{eq:prod-induct}
\pd(G)\leq \max\{\pd(G-N_G[x])+\deg_G(x),\pd(G-x)+1\}.
\end{equation}
\end{lemma}

\begin{proposition}\label{prop:gap1}
For each integer $a\geq 1$, there exists a graph $G_a$ satisfying
\begin{equation}
|G_a|-\ie(G_a)+a< \pd(G_a)< |G_a|-\gamma(G_a)-a.
\end{equation}
\end{proposition}
\begin{proof} Assume that $k$ and $r$ are two positive integers with $k>r+1$. 
Let $T_k:=K_{k,k}$ be the complete bipartite graph and fix a vertex
$v\in V(T_k)$. We take $S_r:=rP_{3s+1}$, which is the disjoint $r$ copies of the path on $3s+1$ vertices ($s\geq 1$) together with a leave vertex $x_i$ from each copy. Now, we from the graph $G_{2k,r}$ on $V(T_k)\cup V(S_r)$ such that $E(G_{2k,r})=E(T_k)\cup E(S_r)\cup \{x_iv\colon i\in [r]\}$. Observe that $\ie(G_{2k,r})=sr+k$ and $\gamma(G_r)=sr+2$. Since $\{K_{k,k}, rP_{3s}\}$ is a projectively prime decomposition of $G_{2k,r}$, we have $2k-1+2rs\leq \pd(G_{2k,r})$ by Theorem~\ref{thm:pd-prime-factor}. Moreover, $\pd(G_{2k,r}-v)=(2k-2)+2rs$ and
$\pd(G_{2k,r}-N_{G_{2k,r}}[v])=2rs$ so that $\pd(G_r)=2k+2rs-1$ by the inequality (\ref{eq:prod-induct}). Consequently, if we choose $k=2r$, we then have 
$\pd(G_{4r,r})-(|G_{4r,r}|-\ie(G_{4r,r}))=r-1$ and $(|G_{4r,r}|-\gamma(G_{4r,r}))-\pd(G_{4r,r})=r-1$.
This in particular implies that if we define $G_a:=G_{4(a+2),a+2}$, then
$$(|G_a|-\ie(G_a))+a<\pd(G_a)<(|G_a|-\gamma(G_a))-a$$ as claimed.
\end{proof}

\begin{proposition}\label{prop:gap2}
For each integer $b\geq 1$, there exists a graph $H_b$ satisfying
\begin{equation}
|H_b|-\gamma(H_b)+b< \pd(H_b)< |H_b|-\epsilon(H_b)-b.
\end{equation}
\end{proposition}
\begin{proof}
We assume that $k$ and $r$ are two positive integers with $k>r+2$.
We take $L_{k,r}:=rK_{k,k}$ and pick a vertex $v_i$ from each copy of $K_{k,k}$. Then, define $R_{k,r}$ to be the graph on $V(L_{k,r})\cup \{x\}$, where $x$ is a vertex disjointly chosen from $V(L_{k,r})$, such that
$E(R_{k,r}):=E(L_{k,r})\cup \{xv_i\colon i\in [r]\}$. Observe that
$|R_{k,r}|=2rk+1$, $\gamma(R_{k,r})=2r$ and $\epsilon(R_{k,r})=r$. Furthermore, we have $\pd(R_{k,r})=(2k-1)r$, which follows from the inequality~(\ref{eq:prod-induct}) together with the fact that $\{rK_{k,k}\}$ is a projectively prime decomposition of $R_{k,r}$.

Secondly, we define another graph $Z_{k,r}$ on $V((r+1)K_{k,k})\cup \{y_2,\ldots,y_{r+1}\}$ by $E(Z_{k,r})=E((r+1)K_{k,k})\cup \{v_1y_j, v_jy_j\colon j\in \{2,3,\ldots,r+1\}\}$, where the vertex $v_l$ is chosen from the $l^{\textrm{th}}$-copy of $K_{k,k}$. Observe that $|Z_{k,r}|=2k(r+1)+r$,
$\gamma(Z_{k,r})=2(r+1)$ and $\epsilon(Z_{k,r})=r+1$. Again, by considering the vertex $v_1$ in the inequality~(\ref{eq:prod-induct}) together with the projectively prime decomposition $\{(r+1)K_{k,k}\}$ of $Z_{k,r}$, we note that $\pd(Z_{k,r})=(2k-1)(r+1)$. 

Now, assuming that we construct $R_{k,r}$ and $Z_{k,r}$
on the disjoint sets of vertices, we form a new graph $H_{k,r}$ on
$V(R_{k,r})\cup V(Z_{k,r})$ by $E(H_{k,r}):=E(R_{k,r})\cup E(Z_{k,r})\cup \{xv_1\}$. It then follows that $|H_{k,r}|=2k(2r+1)+r+1$, $\gamma(H_{k,r})=2(2r+1)$ and $\epsilon(H_{k,r})=2r+1$. Once again, the 
inequality~(\ref{eq:prod-induct}) with respect to the vertex $v_1$ of $H_{k,r}$ together with the projectively prime decomposition $\{rK_{k,k}, (r+1)K_{k,k}\}$, where the first prime arises from $R_{k,r}$ and the second comes up from $Z_{k,l}$ yields that $\pd(H_{k,r})=(2k-1)(2r+1)$.

Finally, if we define $H_b:=H_{b+4,b+1}$, we conclude that
$\pd(H_b)-(|H_b|-\gamma(H_b))=b+1$ and $(|H_b|-\epsilon(H_b))-\pd(H_b)=b+2$ so that the claimed inequality
$$|H_b|-\gamma(H_b)+b< \pd(H_b)< |H_b|-\epsilon(H_b)-b$$ holds.
\end{proof}

Regarding the results of Propositions~\ref{prop:gap1} and~\ref{prop:gap2} together with Theorem~\ref{thm:pd-prime-factor}, the most prominent question would be investigation of the structural properties of projectively prime graphs as well as operations on graphs that preserve projectively primeness. It would be interesting to determine under what conditions on a projectively prime graph $G$, the inequality $\pd(G)\leq |G|-\gamma(G)$ holds. We predict that $C_{3k+1}$-free projectively prime graphs satisfy such an inequality.


\end{document}